\documentclass[runningheads,a4paper]{llncs}
\usepackage{amsmath,amssymb,amsbsy,amsfonts,latexsym,
               amsopn,amstext,amsxtra,euscript,amscd}
\usepackage{graphicx}

\usepackage[top=3.5cm, left=3.2cm, right=3.0cm, bottom=3.0cm]{geometry}

\def\00{{\bf 0}}
\def\+{\oplus}

\def\\{\cr}
\def\({\left(}
\def\){\right)}

\usepackage{url}
\urldef{\mailsa}\path|mehraj.jmi@gmail.com, jamali_dbd@yahoo.co.in, hasan_jmi@yahoo.com|

\begin{document}

\mainmatter  

\title{\sl RICCI CURVATURE FOR SUBMANIFOLDS IN BOCHNER KAHLER MANIFOLD}
\titlerunning{\sl Ricci curvature for submanifolds of Bochner Kahler manifold}

%
\author{ Mehraj Ahmad Lone, Mohammad Jamali, Mohammad Hasan Shahid}
%
\authorrunning{M. A. Lone, M. Jamali, M. H. Shahid}

\institute{Department of Mathematics,\\ Central University of Jammu,\\ Jammu, 180011, INDIA.\\
\vspace{.5cm}
Department of mathematics,\\ Al-Falah University,\\ Faridabad, Haryana, 121004, INDIA.\\
\vspace{.5cm}
Department of mathematics,\\ Jamia Millia Islamia,\\ New Delhi, 110025, INDIA.\\ \mailsa}


 \maketitle
\begin{center}
\textbf{{Abstract}}\end{center}  B. Y. Chen establish the relationship between the Ricci curvature and the squared mean curvature for submanifolds of Riemannian space form with arbitrary codimension. In this paper, we generalize the relationship between the Ricci curvature and the  squared norm of mean curvature vector for submanifolds of Bochner Kahler manifolds.
\vspace{0.3cm}

\noindent \textbf{2010 AMS Classification}: 53C20, 53C21, 53C40.
\vspace{0.2cm}

\section{Introduction}
The relationship between the main extrinsic invariants and the main intrinsic invariants of submanifolds was given by B. Y. Chen\cite{biha4}. This theory was started by Chen in \cite{biha6} and established sharp relationship between the Ricci curvature and the squared norm of mean curvature for a submanifold in a Riemannian  space form with arbitrary codimension \cite{biha5}. In \cite{biha9}, K. Matsumoto \emph {et.al.} obtained  an inequality between the Ricci curvature and the squared norm of mean curvature for submanifolds in complex space form.  In \cite{biha10}, K. Matsumoto \emph{et.al.} obtained the same inequality for the slant submanifolds of complex space form. After that, many research articles \cite{biha1}, \cite{biha7}, \cite{biha11} have been published by different geometers. They obtained the similar inequalities for different submanifolds and ambient spaces in complex as well as in contact version. In this paper, we obtain these inequalities for different submanifolds of Bochner Kahler manifolds.

\section{Preliminaries}
Let $M^{n}$ be a submanifold of a Bochner Kahler manifold $\overline{M}^{2m}$. Let $\nabla$ be the induced Levi-Civita connection on $M$. Then the Gauss and Weingarten formulas are given respectively by
\begin{eqnarray}\label{a1}
\overline{\nabla}_{X}Y = \nabla_{X}Y + h(X,Y),
\end{eqnarray}
\begin{eqnarray}\label{a2}
\overline{\nabla}_{X}V = A_{V}X + D_{X}Y,
\end{eqnarray}
for all $X, Y$ tangent to $M$ and vector field $V$ normal to $M^{n}$. Where $h$, $D$, $A_{V}$ denotes the second fundamental form, normal connection and the shape operator in the direction of $V$. The second fundamental form and the shape operator are related by
\begin{eqnarray}\label{a3}
g(h(X,Y), V) = g(A_{V}X, Y),
\end{eqnarray}
Let $R$ be the curvature tensor of $M^{n}$, Then the Gauss equation is given by
\begin{eqnarray*}\label{a4}
\overline{R}(X,Y,Z,W) = R(X,Y,Z,W) + g(h(X,W),h(Y,Z)) - g(h(X,Z),h(Y,W)),
\end{eqnarray*}
for any vector fields $X$, $Y$, $Z$, $W$ tangent to $M^{n}$.

The curvature tensor of $\overline{M}^{2m}$ is given by \cite{biha12}
\begin{eqnarray}\label{a5}
\overline{R}(X,Y,Z,W)\nonumber &=& L(Y,Z)g(X,W) - L(X,Z)g(Y,W) + L(X,W)g(Y,Z) \\  && - L(Y,W)g(X,Z) + M(X,W)g(JX,W) - M(X,Z)g(JY,W)\\ \nonumber && + M(X,W)g(JY,Z) - M(Y,W)g(JX,Z) \\ \nonumber && - 2M(X,Y)g(JZ,W) - 2M(Z,W)g(JX,Y),
\end{eqnarray}

where
\begin{eqnarray}\label{a6}
L(Y,Z) = \frac{1}{2n+4}Ric(Y,Z) - \frac{\rho}{2(2n+2)(2n+4)}g(Y,Z),
\end{eqnarray}
\begin{eqnarray}\label{a7}
M(Y,Z)  = -L(Y,JZ),
\end{eqnarray}
\begin{eqnarray}\label{a8}
L(Y,Z)  = L(Z,Y),\hspace{1cm} L(Y,Z) = L(JY, JZ), \hspace{1cm} L(Y, JZ) = -L(JY,Z),
\end{eqnarray}
$Ric$ and $\rho$ are the Ricci tensor and scalar curvature of $M$.

Let $p\in M^{n}$ and $\{e_{1}, ... , e_{n}\}$ be an orthonormal basis of the tangent space $T_{p}M$ and  $\{e_{n+1}, ... , e_{2m}\}$ be the orthonormal basis  of $T^{\perp}M$. We denote by $H$, the mean curvature vector at $p$, that is
\begin{eqnarray}\label{a9}
H(p) = \frac{1}{n}\sum_{i=1}^{n}h(e_{i},e_{i}),
\end{eqnarray}
Also, we set
\begin{eqnarray*}\label{a10}
h_{ij}^{r} = g(h(e_{i},e_{j}),e_{r}), \hspace{1cm} i,j \in \{ 1, ... , n\},\hspace{.2cm} r \in \{n+1, ... ,2m\}
\end{eqnarray*}
and
\begin{eqnarray}\label{a11}
\|h\|^{2} = \sum_{i,j=1}^{n}(h(e_{i},e_{j}), h(e_{i},e_{j})).
\end{eqnarray}
For any $p \in M$ and $X \in T_{p}M$, we put $JX = PX + QX$, where $PX$ and $QX$ are the tangential and normal components of $JX$, respectively.

 We denote by
\begin{eqnarray*}\label{a12}
\|P\|^{2} = \sum_{i,j=1}^{n}g^{2}(Pe_{i}, e_{j}),
\end{eqnarray*}
For a Riemannian manifold $M^{n}$, we denote by $K{\pi}$ the sectional curvature of $M$ associated with a plane section $\pi \subset T_{P}M, p\in M$. For an orthonormal basis $\{e_{1}, e_{2}, ..., e_{n}\}$ of the tangent space $T_{p}M$, the scalar curvature $\rho$ is defined by
\begin{eqnarray*}
\rho = \sum_{i<j}K_{ij},
\end{eqnarray*}
where $K_{ij}$ denotes the sectional curvature of the $2$-plane section spanned by $e_{i}$ and $e_{j}$.

We recall that for a submanifold $M$ in a Riemannian manifold, the relative null space of $M$ at a point $p \in M$ is defined by
\begin{eqnarray*}
\mathcal{N}_{p} = \{ X \in T_{p}M | h(X,Y) = 0, \forall Y \in T_{p}M\}.
\end{eqnarray*}

\section{Ricci curvature  and squared norm of mean curvature}
B. Y. Chen established the sharp relationship between the Ricci curvature and the squared norm of mean curvature  for submanifolds in real space form.

In this section, we prove similar inequalities  for submanifolds of Bochner Kahler manifolds
\begin{definition}
A   submanifold $M^{n}$  of a Bochner Kahler manifold $\overline{M}^{2m}$ is said to be a slant submanifold if for any $p \in M$ and any non zero vector $X \in T_{p}M$, the angle between $JX$ and the tangent space $T_{p}M$ is constant. The complex submanifolds and totally real submanifolds are the slant submanifolds with angle $0$ and $\frac{\pi}{2}$ respectively.

\end{definition}
\begin{theorem}
Let $M^{n}$ be a submanifold of a Bochner Kahler manifold $\overline{M}^{2m}$, then

(i) for each unit vector $X \in T_{p}M$, we have
\begin{eqnarray*}
Ric(X) &\leq& \frac{1}{4}n^{2}\|H\|^{2} + \frac{4n^{3} - 12n^{2} - 2n +10 -(3n^{2} - 9n + 3)\|P\|^{2}}{2(2n+2)(2n+4)}\rho  \nonumber \\ && \frac{6}{2n+4}\sum_{2\leq i<j\leq n}Ric(e_{i}, Je_{j})g(e_{i}, Je_{j}) - \frac{3}{2n+4}Ric(e_{i}, Je_{j})g(e_{i}, Je_{j}).
\end{eqnarray*}

(ii) If $H(p) = 0$, the unit tangent vector $X$ at $p$ satisfies the equality if and only if $X \in \mathcal{N}_{p}.$

(iii) The equality case holds identically for all unit tangent vectors at $p$ if and only if either $p$ is totally geodesic point or $n=2$ and $p$ is totally umbilical point.
\end{theorem}
\begin{proof}
(i) Let $X \in T_{p}M$ be a unit tangent vector at $p$. We choose orthonormal basis $\{e_{1}, e_{2}, ... ,e_{n},e_{n+1}, ... , e_{2m}\}$ such that $\{e_{1}, e_{2}, ... ,e_{n}\}$ are tangent to $M$ at $p$ with $e_{1} = X$, then from Gauss equation we have
\begin{eqnarray*}\label{p1}
\nonumber R(X,Y,Z,W) &=& L(Y,Z)g(X,W) - L(X,Z)g(Y,W) + L(X,W)g(Y,Z)
                          \nonumber \\ && - L(Y,W)g(X,Z)  + M(Y,Z)g(JX,W) - M(X,Z)g(JY,W)
                          \nonumber \\ && - M(X,W)g(JY,Z) - M(Y,W)g(JX,Z)  - 2M(X,Y)(JZ,W)
                           \nonumber \\ && - 2M(Z,W)g(JX,Y) + g(h(X,W), h(Y,Z)) - g(h(X,Z), h(Y,W)),
\end{eqnarray*}
for any  X, Y, Z, W $\in$ TM.
\begin{eqnarray*}\label{p2}
\sum_{i,j} R(e_{i}, e_{j}, e_{j}, e_{i}) &=& L(e_{j},e_{j})g(e_{i},e_{i}) - L(e_{i},e_{j})g(e_{j},e_{i}) + L(e_{i},e_{i})g(e_{j},e_{j})
                          \nonumber \\ && - L(e_{j},e_{i})g(e_{i},e_{j})  + M(e_{j},e_{j})g(Je_{i},e_{i}) - M(e_{i},e_{j})g(Je_{j},e_{i})
                          \nonumber \\ && - M(e_{i},e_{i})g(Je_{j},e_{j}) - M(e_{j},e_{i})g(Je_{i},e_{j})  - 2M(e_{i},e_{j})(Je_{j},e_{i})
                           \nonumber \\ && - 2M(e_{j},e_{i})g(Je_{i},e_{j}) + g(h(e_{i},e_{i}), h(e_{j},e_{j})) - g(h(e_{i},e_{j}), h(e_{j},e_{i})),
\end{eqnarray*}
\begin{eqnarray*}\label{p3}
 \hspace{3cm}                       &=& L(e_{j},e_{j})g(e_{i},e_{i}) - L(e_{i},e_{j})g(e_{j},e_{i}) + L(e_{i},e_{i})g(e_{j},e_{j})
                         \nonumber \\ && - L(e_{j},e_{i})g(e_{i},e_{j})  - L(e_{j},Je_{j})g(Je_{i},e_{i}) + L(e_{i},Je_{j})g(Je_{j},e_{i})
                          \nonumber \\ && + L(e_{i},Je_{i})g(Je_{j},e_{j}) + L(e_{j},Je_{i})g(Je_{i},e_{j}) + 2L(e_{i},Je_{j})(Je_{j},e_{i})
                           \nonumber \\ && + 2L(e_{j},Je_{i})g(Je_{i},e_{j}) + g(h(e_{i},e_{i}), h(e_{j},e_{j})) - g(h(e_{i},e_{j}), h(e_{j},e_{i})).
\end{eqnarray*}
Using (\ref{a8}), (\ref{a9}) and (\ref{a11}), we have
\begin{eqnarray*}\label{p4}
\sum_{i,j} R(e_{i}, e_{j}, e_{j}, e_{i})&=&  2nL(e_{i}, e_{i}) - 2L(e_{i}, e_{j})g(e_{i}, e_{j}) + 6L(e_{i}, Je_{j})g(e_{i}, Je_{j})
                                    \nonumber \\ &&  + n^{2}\|H\|^{2} - \|h\|^{2}.
\end{eqnarray*}
 Last equation simplifies to,
\begin{eqnarray}\label{p5}
2\rho  &=&  2(n-1)L(e_{i}, e_{i}) + 6L(e_{i}, Je_{j})g(e_{i}, Je_{j}) + n^{2}\|H\|^{2} - \|h\|^{2}
\end{eqnarray}
 Combining (\ref{a6}) and (\ref{p5}), we have
\begin{eqnarray*}\label{p6}
\hspace{2cm}2\rho  &=& \frac{2(n-1)}{2n+4}Ric(e_{i}, e_{i}) - \frac{2(n-1)\rho}{2(2n+2)(2n+4)}g(e_{i}, e_{i})
        \nonumber \\ &&   + \frac{6}{2n+4} Ric(e_{i}, Je_{j})g(e_{i}, Je_{j}) - \frac{6 \rho}{2(2n+2)(2n+4)} g(e_{i}, Je_{j})g(e_{i}, Je_{j})
        \nonumber \\ &&   + n^{2}\|H\|^{2} - \|h\|^{2}.
\end{eqnarray*}
Or,
\begin{eqnarray*}\label{p9}
\hspace{2cm}2\rho  &=& \frac{6n^{2} + 2n - 8 - 6\|P\|^{2}}{2(2n + 2)(2n + 4)}\rho + \frac{6}{2n+4} Ric(e_{i}, Je_{j})g(e_{i}, Je_{j})
       \nonumber \\ && + n^{2}\|H\|^{2} - \|h\|^{2}.
\end{eqnarray*}
Or,
\begin{eqnarray*}
n^{2}\|H\|^{2} = (2-\frac{6n^{2}+ 2n -8 -6\|P\|^{2}}{2(2n+2)(2n+4)})\rho + \|h\|^{2} - \frac{6}{2n+4}Ric(e_{i},Je_{j})g(e_{i}, Je_{j}).
\end{eqnarray*}
From which we have,
\begin{eqnarray*}
n^{2}\|H\|^{2} = (2-\frac{6n^{2}+ 2n -8 -6\|P\|^{2}}{2(2n+2)(2n+4)})\rho + \sum_{r=n+1}^{2m}\sum_{i,j=1}^{n}(h_{ij}^{r})- \frac{6}{2n+4}Ric(e_{i},Je_{j})g(e_{i}, Je_{j}).
\end{eqnarray*}
Or,
\begin{eqnarray*}
n^{2}\|H\|^{2} &=& (2-\frac{6n^{2}+ 2n -8 -6\|P\|^{2}}{2(2n+2)(2n+4)})\rho + \sum_{r=n+1}^{2m}\left[(h_{11}^{r})^{2} + (h_{22}^{r})^{2}+ ... + (h_{nn}^{r})^{2}  + 2\sum_{i<j}(h_{ij}^{r})^{2}\right] \\ && - \frac{6}{2n+4}Ric(e_{i},Je_{j})g(e_{i}, Je_{j}).
\end{eqnarray*}
Which implies that
\begin{eqnarray*}
n^{2}\|H\|^{2} &=& (2-\frac{6n^{2}+ 2n -8 -6\|P\|^{2}}{2(2n+2)(2n+4)})\rho + \sum_{r=n+1}^{2m}\left[(h_{11}^{r})^{2} + (h_{22}^{r})^{2}+ ... + (h_{nn}^{r})^{2}\right]\\ && + 2\sum_{r=n+1}^{2m}\sum_{i<j}(h_{ij}^{r})^{2}  - \frac{6}{2n+4}Ric(e_{i},Je_{j})g(e_{i}, Je_{j}).
\end{eqnarray*}
Or,
\begin{eqnarray*}
n^{2}\|H\|^{2} &=& (2-\frac{6n^{2}+ 2n -8 -6\|P\|^{2}}{2(2n+2)(2n+4)})\rho + \sum_{r=n+1}^{2m}\left[(h_{11}^{r})^{2} + (h_{22}^{r}+ ... + h_{nn}^{r})^{2} - 2 \sum_{2\leq i<j\leq n} h_{ii}^{r}h_{jj}^{r}\right]\\ && + 2\sum_{r=n+1}^{2m}\sum_{i<j}(h_{ij}^{r})^{2} - \frac{6}{2n+4}Ric(e_{i},Je_{j})g(e_{i}, Je_{j}).
\end{eqnarray*}
Form which we obtain,
\begin{eqnarray*}
n^{2}\|H\|^{2} &=& (2-\frac{6n^{2}+ 2n -8 -6\|P\|^{2}}{2(2n+2)(2n+4)})\rho + \frac{1}{2}\sum_{r=n+1}^{2m}\left[(h_{11}^{r} + h_{22}^{r}+ ... + h_{nn}^{r})^{2} + (h_{11}^{r} - h_{22}^{r} - ... - h_{nn}^{r})^{2}\right]\\ && -  2\sum_{r=n+1}^{2m}\sum_{2\leq i<j\leq n} h_{ii}^{r}h_{jj}^{r} + 2\sum_{r=n+1}^{2m}\sum_{i<j}(h_{ij}^{r})^{2} - \frac{6}{2n+4}Ric(e_{i},Je_{j})g(e_{i}, Je_{j}).
\end{eqnarray*}
Or,
\begin{eqnarray*}
n^{2}\|H\|^{2} &=& (2-\frac{6n^{2}+ 2n -8 -6\|P\|^{2}}{2(2n+2)(2n+4)})\rho + \frac{1}{2}\sum_{r=n+1}^{2m}(h_{11}^{r} + h_{22}^{r}+ ... + h_{nn}^{r})^{2} + \frac{1}{2}\sum_{r=n+1}^{2m}(h_{11}^{r} - h_{22}^{r} - ... - h_{nn}^{r})^{2}\\ && -  2\sum_{r=n+1}^{2m}\sum_{2\leq i<j\leq n} h_{ii}^{r}h_{jj}^{r} + 2\sum_{r=n+1}^{2m}\sum_{j=1}^{n}(h_{1j}^{r})^{2} + 2\sum_{r=n+1}^{2m}\sum_{i<j}(h_{ij}^{r})^{2} - \frac{6}{2n+4}Ric(e_{i},Je_{j})g(e_{i}, Je_{j}).
\end{eqnarray*}
Or,
\begin{eqnarray}\label{p1}
n^{2}\|H\|^{2}\nonumber &=& (2-\frac{6n^{2}+ 2n -8 -6\|P\|^{2}}{2(2n+2)(2n+4)})\rho + \frac{1}{2}\sum_{r=n+1}^{2m}(h_{11}^{r} + h_{22}^{r}+ ... + h_{nn}^{r})^{2} + \frac{1}{2}\sum_{r=n+1}^{2m}(h_{11}^{r} - h_{22}^{r} - ... - h_{nn}^{r})^{2} \\ && +  2\sum_{r=n+1}^{2m}\sum_{j=1}^{n}(h_{1j}^{r})^{2} - 2\left[ \sum_{r=n+1}^{2m}\sum_{2\leq i<j\leq n} h_{ii}^{r}h_{jj}^{r} - (h_{ij}^{r})^{2}\right] - \frac{6}{2n+4}Ric(e_{i},Je_{j})g(e_{i}, Je_{j}).
\end{eqnarray}
From Gauss equation, we have
\begin{eqnarray*}
K_{ij} = 2L(e_{i}, e_{i}) + 6L(e_{i},e_{i})g(e_{i}, Je_{j}) + \sum_{r=n+1}^{2m}\left[ h_{ii}^{r}h_{jj}^{r} - (h_{ij}^{r})^{2}\right]
\end{eqnarray*}
\begin{eqnarray*}
&& = \frac{2}{2n+4}Ric(e_{i}, e_{i}) - \frac{2\rho}{2(2n+2)(2n+4)}g(e_{i},e_{i}) + \frac{6}{2n+4}Ric(e_{i}, Je_{j})g(e_{i},Je_{j})  \\ && -\frac{6\rho}{2(2n+2)(2n+4)}g(e_{i},Je_{j})g(e_{i},Je_{j}) + \sum_{r=n+1}^{2m}\left[ h_{ii}^{r}h_{jj}^{r} - (h_{ij}^{r})^{2}\right]
\end{eqnarray*}
\begin{eqnarray*}
&& = \frac{8n + 6 - 6\|p\|^{2}}{2(2n+2)(2n+4)}\rho + \frac{6}{2n+4}Ric(e_{i}, Je_{j})g(e_{i}, Je_{j}) + \sum_{r=n+1}^{2m}\left[h_{ii}^{r}h_{jj}^{r} - (h_{ij}^{r})^{2}\right]
\end{eqnarray*}
and consequently
\begin{eqnarray}\label{p2}
\sum_{2\leq i<j\leq n} K_{ij} \nonumber &=& \frac{4n^{3}-9n^{2}-n+6 - (3n^{2} - 9n + 6)\|p\|^{2}}{2(2n+2)(2n+4)}\rho +\frac{6}{2n+4} \sum_{2\leq i<j\leq n}Ric(e_{i}, Je_{j})g(e_{i}, Je_{j})\\ &&  + \sum_{r=n+1}^{2m}\sum_{2\leq i<j\leq n}\left[ h_{ii}^{r}h_{jj}^{r} - (h_{ij}^{r})^{2}\right]
\end{eqnarray}
Incorporating (\ref{p2}) in (\ref{p1}), we get
\begin{eqnarray*}\label{p3}
\frac{1}{2}n^{2}\|H\|^{2} &\geq& 2Ric(X) - \frac{6n^{2} + 2n -8 -6\|P\|^{2}}{2(2n+2)(2n+4)}\rho + 2\frac{4n^{3} - 9n^{2} - n +6- (3n^{2} - 9n + 6)\|P\|^{2}}{2(2n+2)(2n+4)}\rho  \nonumber \\ && \frac{12}{2n+4}\sum_{2\leq i<j\leq n}Ric(e_{i}, Je_{j})g(e_{i}, Je_{j}) - \frac{6}{2n+4}Ric(e_{i}, Je_{j})g(e_{i}, Je_{j})
\end{eqnarray*}
or,
\begin{eqnarray*}
Ric(X) &\leq& \frac{1}{4}n^{2}\|H\|^{2} + \frac{4n^{3} - 12n^{2} - 2n +10 -(3n^{2} + 9n -3)\|P\|^{2}}{2(2n+2)(2n+4)}\rho  \nonumber \\ && \frac{6}{2n+4}\sum_{2\leq i<j\leq n}Ric(e_{i}, Je_{j})g(e_{i}, Je_{j}) - \frac{3}{2n+4}Ric(e_{i}, Je_{j})g(e_{i}, Je_{j})
\end{eqnarray*}
(ii) Suppose $H(p) = 0$, equality holds if and only if

\vspace{.4cm}
$\begin{cases}
h_{12}^{r} = ... = h_{1n}^{r} = 0, \\
h_{11}^{r} = h_{22}^{r}+ ...+h_{nn}^{r}, r \in\{n+1, ...,2m\}
\end{cases}$
\vspace{.4cm}

Then $h_{1j}^{r} = 0 \forall j \in \{1, 2, ...,n\}, r \in\{n+1, ...,2m\},$ i.e. $X \in\mathcal{N}$.\newline
(iii) The equality case holds for all unit vectors at $p$ if and only if

\vspace{.4cm}
$\begin{cases}
h_{12}^{r} = 0, i\neq j, r \in \{n+1, ..., 2m\},\\
h_{11}^{r}+ ...+h_{nn}^{r} - 2h_{ij}^{r} = 0, i\in\{1, 2, ...,n\}, r \in\{n+1, ...,2m\}
\end{cases}$
\vspace{.4cm}

We distinguish two cases:

(a) $n\neq 2$, then $p$ is a totally geodesic point

(b) n=2, it follows that $p$ is a totally   umbilical point.

 The converse is trivial.
\end{proof}

\begin{corollary}
Let $M^{n}$ be a submanifold of a Bochner Kahler manifold $\overline{M}^{2m}$ which is Einstein, then

(i)for each unit vector $X \in T_{p}M$, we have
\begin{eqnarray*}
Ric(X) &\leq& \frac{1}{4}n^{2}\|H\|^{2} + \frac{4n^{3} - 12n^{2} - 2n +10 -(3n^{2} - 9n + 3)\|P\|^{2}}{2(2n+2)(2n+4)}\rho  \nonumber \\ &&
+\frac{3}{2n+2}\lambda \|P\|^{2}.
\end{eqnarray*}
(ii) If $H(p) = 0$, the unit tangent vector $X$ at $p$ satisfies the equality if and only if $X \in \mathcal{N}_{p}.$ \newline
(iii) The equality case holds identically for all unit tangent vectors at $p$ if and only if either $p$ is totally geodesic point or $n=2$ and $p$ is totally umbilical point.
\end{corollary}
\begin{theorem}
Let $M^{n}$ be a slant submanifold of a Bochner Kahler manifold $\overline{M}^{2m}$, then

(i) for each unit vector $X \in T_{p}M$, we have
\begin{eqnarray*}
Ric(X) &\leq& \frac{1}{4}n^{2}\|H\|^{2} + \frac{4n^{3} - 12n^{2} - 2n +10 -(3n^{2} - 9n + 3)cos^{2}\theta}{2(2n+2)(2n+4)}\rho  \nonumber \\ && + \frac{6cos\theta}{2n+4}\sum_{2\leq i<j\leq n}Ric(e_{i}, Je_{j}) - \frac{3}{2n+4}cos\theta Ric(e_{i}, Je_{j})\end{eqnarray*}

(ii) If $H(p) = 0$, the unit tangent vector $X$ at $p$ satisfies the equality if and only if $X \in \mathcal{N}_{p}.$

(iii) The equality case holds identically for all unit tangent vectors at $p$ if and only if either $p$ is totally geodesic point or $n=2$ and $p$ is totally umbilical point.
\end{theorem}
\begin{corollary}
Let $M^{n}$ be an anti-invariant submanifold of a Bochner Kahler manifold $\overline{M}^{2m}$, then
(i)for each unit vector $X \in T_{p}M$, we have
\begin{eqnarray*}
Ric(X) &\leq& \frac{1}{4}n^{2}\|H\|^{2} + \frac{4n^{3} - 12n^{2} - 2n +10}{2(2n+2)(2n+4)}\rho
\end{eqnarray*}
(ii) If $H(p) = 0$, the unit tangent vector $X$ at $p$ satisfies the equality if and only if $X \in \mathcal{N}_{p}.$ \newline
(iii) The equality case holds identically for all unit tangent vectors at $p$ if and only if either $p$ is totally geodesic point or $n=2$ and $p$ is totally umbilical point.
\end{corollary}
\begin{corollary}
Let $M^{n}$ be a invariant submanifold of a Bochner Kahler manifold $\overline{M}^{2m}$, then

(i) for each unit vector $X \in T_{p}M$, we have
\begin{eqnarray*}
Ric(X) &\leq& \frac{1}{4}n^{2}\|H\|^{2} + \frac{4n^{3} - 12n^{2} - 2n +10 -(3n^{2} - 9n + 3)}{2(2n+2)(2n+4)}\rho   \nonumber \\ && + \frac{6}{2n+4}\sum_{2\leq i<j\leq n}Ric(e_{i}, Je_{j}) - \frac{3}{2n+4} Ric(e_{i}, Je_{j})
\end{eqnarray*}

(ii) If $H(p) = 0$, the unit tangent vector $X$ at $p$ satisfies the equality if and only if $X \in \mathcal{N}_{p}.$

(iii) The equality case holds identically for all unit tangent vectors at $p$ if and only if either $p$ is totally geodesic point or $n=2$ and $p$ is totally umbilical point.
\end{corollary}
%

\end{document}